\documentclass[11pt]{article}

\usepackage{amssymb, amsmath, amsthm, graphicx}
\usepackage{caption}
\usepackage{subcaption}
\usepackage{tikz}

\usepackage{ifthen}

\usepackage[left=1in,top=1in,right=1in]{geometry}
\date{}

\theoremstyle{plain}
      \newtheorem{theorem}{Theorem}[section]
      \newtheorem{lemma}[theorem]{Lemma}
            \newtheorem{claim}[theorem]{Claim}
      \newtheorem{problem}[theorem]{Problem}
      
      \newtheorem{corollary}[theorem]{Corollary}
      \newtheorem{proposition}[theorem]{Proposition}
      
\theoremstyle{definition}

\theoremstyle{remark}
      \newtheorem{remark}[theorem]{Remark}

\DeclareMathOperator{\pr}{Pr}
\DeclareMathOperator{\area}{area}

\title{Disjoint faces in simple drawings of the complete graph and topological Heilbronn problems}

\begin{document}

\author{Alfredo Hubard\thanks{LIGM, Universit\'e Gustave Eiffel, CNRS, ESIEE Paris, F-77454 Marne-la-Vall\'ee, France. Email: {\tt alfredo.hubard@univ-eiffel.fr}}
\and Andrew Suk\thanks{Department of Mathematics, University of California at San Diego, La Jolla, CA, 92093 USA. Supported by NSF CAREER award DMS-1800746 and  NSF award DMS-1952786. Email: {\tt asuk@ucsd.edu}.}}

\maketitle

\begin{abstract}
 Given a complete simple topological graph $G$, a $k$-face generated by $G$ is the open bounded region enclosed by the edges of a non-self-intersecting $k$-cycle in $G$.  Interestingly, there are complete simple topological graphs with the property that every odd face it generates contains the origin.  In this paper, we show that every complete $n$-vertex simple topological graph generates at least $\Omega(n^{1/3})$ pairwise disjoint 4-faces.  As an immediate corollary, every complete simple topological graph on $n$ vertices drawn in the unit square generates a 4-face with area at most $O(n^{-1/3})$. Finally, we investigate a $\mathbb Z_2$ variant of Heilbronn triangle problem.  

\end{abstract}

\section{Introduction}

A \emph{topological graph} is a graph drawn in the plane such that its vertices are represented by points and its edges are represented by non-self-intersecting arcs connecting the corresponding points. The arcs are not allowed to pass through vertices different from their endpoints, and if two edges share an interior point, then they must properly cross at that point in common. A topological graph is \emph{simple} if every pair of its edges intersect at most once, either at a common endpoint or at a proper crossing point.  If the edges are drawn as straight-line segments, then the graph is said to be \emph{geometric}. Simple topological graphs have been extensively studied \cite{Ai,SZ,PST,FV,H}, and are sometimes referred to as simple drawings \cite{H,Ar}.  In this paper, we study the crossing pattern of the faces generated by a simple topological graph.

 If $\gamma\subset \mathbb R^2$ is a {\em Jordan curve} (i.e.~non-self-intersecting closed curve), then by the Jordan curve theorem, $\mathbb R^2 \setminus \gamma$ has two connected components one of which is bounded. For any Jordan curve $\gamma\subset \mathbb{R}^2$, we refer to the bounded open region of $\mathbb R^2 \setminus \gamma$ given by the Jordan curve theorem as the {\em face inside} of $\gamma$. We refer to the {\em area} of $\gamma$ as the area of the face inside of $\gamma$, which we denote by $\area(\gamma)$.

 It is known that every complete simple topological graph $G$ of $n$ vertices contains many non-self-intersecting $k$-cycles, for $k=(\log n)^{1/4 - o(1)} $ (e.g.~see \cite{PST,SZ,PR,MT}). A \emph{$k$-face} generated by $G$ is the face inside of a non-self-intersecting $k$-cycle in $G$.  For simplicity, we say that a $k$-face is \emph{in} $G$, if $G$ generates it, and we call it an \emph{odd} (even) face if $k$ is odd (even).  Let us remark that a $k$-face in $G$ may contain other vertices and edges from $G$.  Moreover, notice that if $G$ is simple then every 3-cycle in $G$ must be non-self-intersecting, so for convenience, we call 3-faces \emph{triangles}.  

Surprisingly, one cannot guarantee two disjoint $3$-faces in complete simple topological graphs.  In the next section, we will show that the well-known construction due to Harborth and Mengerson \cite{HM}, known as the twisted graph, shows the following.

 \begin{proposition}\label{twisted}
For every $n \geq 1$, there exists a complete $n$-vertex simple topological graph such that every odd face it generates contains the origin.

 \end{proposition}

\noindent See Figure 1.  However, the main result in this paper shows that we can guarantee many pairwise disjoint 4-faces.

\begin{figure}\label{figtwist}
\centering
\rotatebox{-45}{
\begin{tikzpicture}[scale=0.5]
\tikzstyle{every node}=[circle, draw, fill=black!75,inner sep=0pt, minimum width=4pt]
\foreach \i in {0,...,3}
	{
	\draw [red, thick] plot [smooth]  coordinates{(0,0)(\i/2 ,10-\i/2-.1 )(9.9-\i/2 ,10-\i/2)(9.7-\i/2 ,\i )(\i/2 +1/2,\i/2 +1/2)};
	\ifthenelse{\i>0} {\draw [blue,thick] plot  [smooth] coordinates{(1/2,1/2)(\i*.63+4/3 ,10-\i*2/3 )(10-\i*2/3 ,10-\i*2/3)(9.9-\i*2/3 ,\i*4/3 )(\i/2 +1/2,\i/2 +1/2)}}{};
	\ifthenelse{\i>1} {\draw [green,thick] plot [smooth]  coordinates{(1,1)(\i*9/12+2.3 ,9.9-\i*2/3 )(9.9-\i*2/3 ,9.9-\i*2/3)(9.9-\i*2/3  ,\i*4/3+.3 )(\i/2 +1/2,\i/2 +1/2)}}{};

	}
	\draw [black,thick] plot [smooth]  coordinates{(1.5,1.5)(5.3, 7.7 )(7.7,7.7)(7.7,4.5)(2,2)};
	\foreach \x in {0,...,4}
	{\node at (\x/2,\x/2){};
	}

\end{tikzpicture}}
\caption{The complete twisted graph on 5 vertices.}
\end{figure}
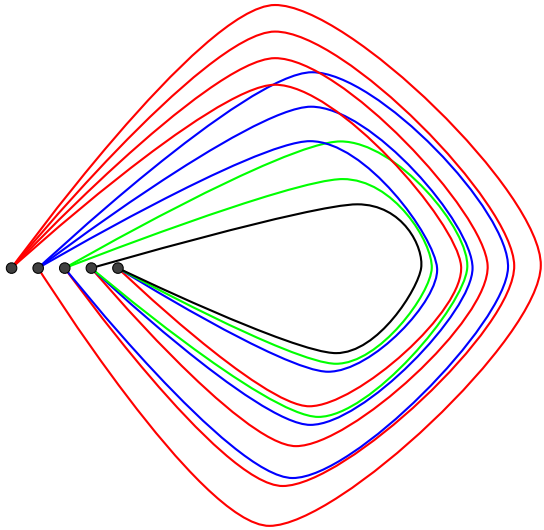

%


\begin{theorem}\label{main}

Every $n$-vertex complete simple topological graph generates at least $\Omega(n^{1/3})$ pairwise disjoint 4-faces.
 
\end{theorem}

We apply the results mentioned above to a topological variant of Heilbronn's triangle problem.  Over 70 years ago, Heilbronn asked: What is the smallest $h(n)$ such that any set of $n$ points in the unit square spans a triangle whose area is at most $h(n)$?   A simple triangulation argument shows that $h(n) \leq O(\frac{1}{n})$.  This was improved several times by Roth and Schmidt \cite{roth1,roth2,roth3,roth4,schmidt}, and currently, the best known upper bound is $\frac{1}{n^{8/7 - o(1)}}$ due to Koml\'os, Pintz, and Szemer\'edi~\cite{kpstriangle}.  Heilbronn conjectured that $h(n) = \Theta(\frac{1}{n^2})$, which was later disproved by Koml\'os, Pintz, and Szemer\'edi~\cite{kpslower}, who showed that $h(n) \geq \Omega(\frac{\log n}{n^2})$.  Erd\H os \cite{erdosH} conjectured that this new bound is asymptotically tight.  

Here, we study Heilbronn's problem for topological graphs.  A simple variant of Proposition~\ref{twisted} shows that one cannot guarantee a small triangle in complete simple topological graphs drawn in the unit square. 

 \begin{proposition}\label{twisted2}
 
 For every $n \geq 1$ and $\varepsilon > 0$, the complete $n$-vertex simple topological graph can be drawn in the unit square such that every odd face it generates has area at least $1-\varepsilon$.
 \end{proposition}

On the other hand, as an immediate corollary to Theorem \ref{main},  we have the following.

\begin{corollary}
 Every $n$-vertex complete simple topological graph drawn in the unit square generates a 4-face with area at most $O(\frac{1}{n^{1/3}})$.
\end{corollary}
 
 \noindent In the other direction, a construction due to Lefmann \cite{lefmann} shows that the complete $n$-vertex geometric graph can be drawn in the unit square such that every 4-face has area at least $\Omega\left(\frac{\log^{1/2}n}{n^{3/2}}\right)$.  It would be interesting to see if one can improve this bound for simple topological graphs.  Lastly, let us mention that Heilbronn's triangle problem has been studied for $k$-gons, and we refer the interested reader to \cite{lefmann} for more results.

Our paper is organized as follows.  In Section \ref{const}, we establish Propositions \ref{twisted} and \ref{twisted2}. In Section \ref{4face}, we establish a lemma on finding 4-faces in complete simple topological graph.  In Section \ref{maindis}, we use this lemma to prove Theorem \ref{main}.  Finally in Section \ref{last}, we consider Heilbronn's triangle problem for (not necessarily simple) topological graphs.

 
\section{The complete twisted graph}\label{const}

 The \emph{complete twisted graph} on $n$ vertices is a complete simple topological graph with vertices labelled $1$ to $n$ which we will draw on the horizontal axis from left to right, with the property that two edges intersect if their indices are nested, i.e., edges $(i,j)$ and $(k,\ell)$, with $i<j$, $k<\ell$, intersect if and only if $i < k<\ell < j$ or $k < i < j<\ell$.  See Figure 1.  The complete twisted graph was introduced by Harborth and Mengerson \cite{HM} as an example of a complete simple topological graph with no subgraph that is weakly isomorphic to the complete convex geometric graph on five vertices.  See also \cite{PST,SZ} for more applications.

\begin{proposition}\label{propt}
There exists a common point in the interior of all the odd faces generated by the complete twisted graph. Moreover, for every $\varepsilon>0$, the complete twisted graph can be drawn in the unit square such that every odd face has area at least $1-\varepsilon$. 
\end{proposition}

We will need the following claim, which is essentially equivalent to the Jordan curve theorem for piece-wise smooth curves.

\begin{claim}\label{parity}
Let $\gamma$ be a piece-wise smooth Jordan curve contained in an open Euclidean disk $D$. Then a point $p$ is in the bounded region of $\mathbb R^2 \setminus \gamma$ if there is a piece-wise linear arc $\alpha$ starting at $p$ and ending in  $\mathbb R^2 \setminus D$ such that $\alpha$ (properly) crosses $\gamma$ an odd number of times.
\end{claim}
\begin{proof}(Sketch)
 It can be shown that perturbing $\alpha$ locally doesn't change the parity of the number of intersections between
 $\gamma$ and $\alpha$.  That is, proper crossings can only appear or disappear in pairs when $\alpha$ is
perturbed.  See Figure \ref{jord}.

\begin{figure}
\centering
\begin{tikzpicture}

\draw [blue] plot  coordinates{(.5,1.5) (1/2,1/2) (1,1) (3/2,1/2)}; 

\draw plot coordinates {(0,3/4) (5/3,3/4)};

\draw [blue] plot coordinates{(3.5,1.5) (1/2+3,1/2) (1+3,1) (3+3/2,1/2)}; 

\draw plot coordinates {(1+2,5/4) (1+5/3+2,5/4)};

\end{tikzpicture}
\caption{Parity of proper crossings under a perturbation of the arc, a local picture.}
\end{figure}
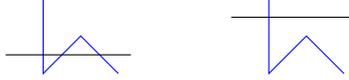\label{jord}  Therefore, in what follows we denote by $lk_2(p,\gamma)$ the parity of the number of intersections between $\gamma$ and any arc $\alpha$ that crosses $\gamma$ properly, starting at $p$, ending in a point outside of $D$. Notice that if $p$ and $q$ are connected by a piece-wise smooth arc that avoids $\gamma$,  then $lk_2(p,\gamma)=lk_2(q,\gamma)$. Consider for every pair of points $p$ and $q$, two rays $\alpha_p$ and $\alpha_q$ that emanate from $p$ and $q$ respectively, and nearly reach $\gamma$. Assume that both stay at some small positive distance $\epsilon>0$ from $\gamma$.  Then extend $\alpha_p$ following $\gamma$ closely without ever intersecting it, 
since $\gamma$ has finite length, if $p$ and $q$ are in the same connected component then the extension of $\alpha_p$ eventually intersects $\alpha_q$ and together they form an arc connecting $p$ and $q$ that avoids $\gamma$. 
Furthermore, two points that lie near $\gamma$ on opposite sides of $\gamma$ have different parity, so we can conclude that the two path-connected components of $\mathbb R^2 \setminus \gamma$ can be identified with the two possible values of $lk_2(.,\gamma)$. Finally, observe that for any point $p$ outside $D$,
any arc that stays outside $D$ doesn't intersect $\gamma$, hence $lk_2(p,\gamma)=0$.  Moreover, a point is in the bounded component of $\mathbb R^2 \setminus \gamma$ if and only if $lk_2(p,\gamma)=1$. \end{proof}

\begin{proof}[Proof of Proposition \ref{propt}]
Consider the complete $n$-vertex twisted graph such vertex $v_i$ is placed at $(i,0)$.  See Figure 1. Let $p=(n+1,0)$ consider a ray emanating out of $p$ that passes just above the vertices. This ray intersects each edge of the twisted drawing exactly once.  Hence, for any non-self intersecting odd cycle $\gamma$ in $G$, $lk_2(p,\gamma)=1$.  By Claim \ref{parity}, $p$ is in the face inside $\gamma$. 

To upgrade this drawing so that each odd face has large area, we can apply a homeomorphism $\phi$ to the plane such that the drawing lies in the unit square, all the vertices cluster around the origin, and each face that contains $\phi(p)$ has area at least $1-\varepsilon$.
\end{proof}

\section{Finding a 4-face inside a large face}\label{4face}

In this section, we establish several lemmas that will be used in the proof of Theorem \ref{main}.  First, let us clarify some terminology.  Given a \emph{planar} graph $H$ drawn in the plane (with no crossing edges), the components of the complement of $H$ are called the \emph{faces} of $H$.  Let $G$ be a complete simple topological graph and let $T$ be a triangle in $G$.  We let $V(T)$ denote the set of vertices of the 3-cycle in $G$ that generates $T$.  We say that $T$ is \emph{incident} to vertex $v \in V(G)$, if $v \in V(T)$.  We say that triangle $T$ is \emph{empty}, if there is no vertex from $G$ that lies in $T$.  We will repeatedly use the following lemma due to Ruiz-Vargas (see also \cite{FV}).

\begin{lemma}[\cite{RV}]\label{mex}
Let $G$ be a complete simple topological graph and $H$
be a connected plane subgraph of $G$ with at least two vertices. Let
$v$ be a vertex of $G$ that is not in $H$, and let $F$ be the face of $H$ that
contains $v$.  Then there exist two edges of $G$ emanating out of $v$ to the boundary of $F$ such that their interior lies complete inside of $F$.

\end{lemma}

\noindent We omit the proofs of the following two lemmas which are simple consequences of Lemma \ref{mex}.

\begin{lemma}\label{one}
Let $G$ be a complete simple topological graph and $H$
be a connected plane subgraph of $G$ with at least two vertices. Let
$v$ be a vertex of $H$ with degree one, and let $F$ be the face of $H$ whose boundary contains $v$.  Then there exist an edge of $G$ emanating out of $v$ to the boundary of $F$ such that its interior lies complete inside of $F$.

\end{lemma}

\begin{lemma}\label{triangle}

Let $G$ be a complete simple topological graph on four vertices, and let $T$ be a triangle in $G$ with a vertex $v \in V(G)$ inside of it.  Then $G$ generates a 4-face that lies inside of the triangle $T$.  

\end{lemma}

Lastly, we will need following key lemma, which can be considered as a generalization of Lemma~\ref{triangle}. Given a plane graph $H$ and a face $F$ in $H$, the \emph{size} of $F$, denoted by $|F|$, is the total length of the closed walk(s) in $H$ bounding the face $F$.  Given two vertices $u,v$ along the boundary of $F$, the \emph{distance} between $u$ and $v$ is the length of the shortest walk from $u$ to $v$ along the boundary of $F$.

\begin{lemma}\label{key}
Let $k \geq 5$ and $G$ be a complete simple topological graph and $H$
be a connected plane subgraph of $G$ with minimum degree two. Let $F$ be a face of $H$ such that $|F|= k$ and contains at least $6(k-4)$ vertices of $G$ in its interior. Then $G$ generates a 4-face that lies inside of $F$.  

\end{lemma}

\begin{proof}

We proceed by induction on $k$, the size of $F$.  For the base case $k = 5$, since $H$ has minimum degree two, the boundary of $F$ must be simple $5$-cycle.  Let $v_1,\ldots, v_5$ be the vertices along the boundary of $F$ appearing in clockwise order.  Let $u_1,\ldots ,u_6$ be the vertices of $G$ in the interior of $F$.  By applying Lemma \ref{mex} to $u_i$ and the plane graph $H$, we obtain two edges emanating out of $u_i$ to the boundary of $F$, whose interior lies completely inside of $F$. If the endpoints of these edges have distance more than one along the boundary of $F$, then we have generated a 4-face inside of $F$ and we are done.  Therefore, we can assume that for each $u_i$, the two edges emanating out of it obtained from Lemma \ref{mex} have endpoints at distance one (consecutive) along the boundary of $F$. 

Since $|F| = 5$, by the pigeonhole principle, there are two vertices, say $u_1$ and $u_2$, such that the two edges emanating out of $u_1$ and $u_2$ obtained from Lemma \ref{mex} go to the same two consecutive vertices, say $v_1,v_{2}$.  If these 4 edges are non-crossing, then we obtain a triangle with a vertex inside of it.  See Figure \ref{fig:inside1}.  By Lemma \ref{triangle}, we obtain a 4-face inside of $F$ and we are done.  Therefore, without loss of generality, we can assume that edges $u_2v_1$ and $u_1v_2$ cross. 

Let $H' = H\cup\{u_1v_1,u_1v_2,u_2v_2\}$, and let $F'$ be the face such that $u_2$ lies on the boundary of $F'$.   See Figure \ref{fig:inside2}.  Since $u_2$ has degree one in $H'$, we apply Lemma \ref{one} to $u_2$ and $H'$ to obtain an edge $u_2v_i$ emanating out of $u_2$ to the boundary of $F'$, whose interior lies in $F'$.

\begin{figure} 
     \centering
     \begin{subfigure}[b]{0.25\textwidth}
         \centering
         \includegraphics[width=\textwidth]{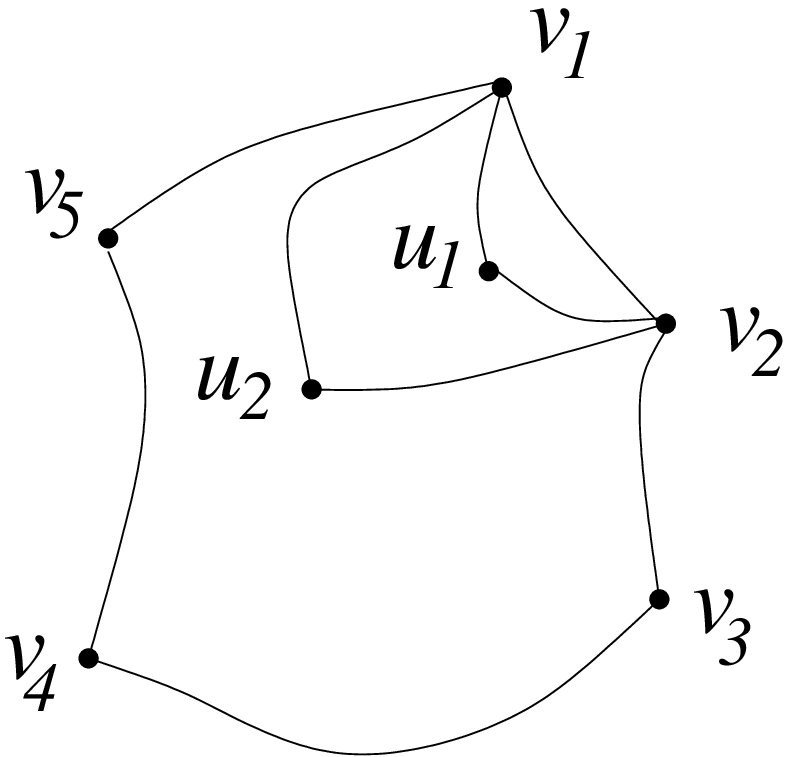}
         \subcaption{Non-crossing edges.}
         \label{fig:inside1}
     \end{subfigure}
     \hspace{2cm}
     \begin{subfigure}[b]{0.25\textwidth}
         \centering
         \includegraphics[width=\textwidth]{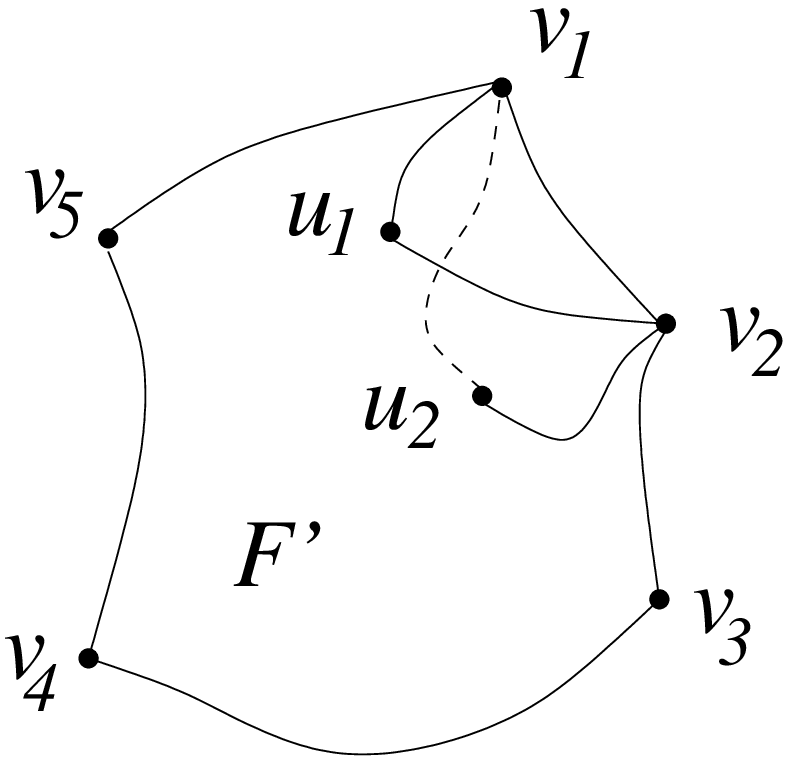}
         \subcaption{Plane subgraph $H'$.}
         \label{fig:inside2}
     \end{subfigure}
        \caption{Finding a 4-face inside a 5-face.}
        \label{fig:inside12}
\end{figure}

If $v_i = v_3$, then we obtain a 4-face inside of $F$ by following the sequence of vertices $(v_3,u_2,v_1,v_2)$ in $G$.  If $v_i = v_4$, then we obtain a 4-face inside of $F$ by following the sequence vertices $(v_4,u_2,v_2,v_3)$ in $G$.   If $v_i = v_5$, then we obtain the 4-face inside of $F$ by following sequence vertices $(v_5,v_1,v_2,u_2)$ in $G$.  Finally, if $v_i = u_1$, then by following the sequence vertices $(u_2,u_1,v_1,v_2)$ in $G$, we obtain a 4-face inside of $F$.  

For the inductive step, assume that the statement holds for all $k' < k$.  Let $F$ be a face of $H$ such that $|F| = k$, and let $(v_1, v_2,\ldots, v_k,v_1)$ be the closed walk(s) along the entire boundary of $F$.  Set $t = 6(k-4)$, and let $u_1,\ldots, u_t$ be the vertices of $G$ that lie in the interior of $F$.  For each $u_i$, we apply Lemma \ref{mex}, with respect to $H$, to obtain two edges emanating out of $u_i$ to the boundary of $F$, such that their interior lies inside of $F$.  The proof now falls into the following cases.

\medskip

\noindent \emph{Case 1.}  Suppose there is a $u_i$ such that the two edges emanating out of $u_i$ obtained from Lemma~\ref{mex} have endpoints at distance two along the boundary of $F$.   Then we have created a 4-face inside of $F$ and we are done. 

\medskip

\noindent \emph{Case 2.}  Suppose there is a vertex $u_i$ such that the two edges emanating out of $u_i$ obtained from Lemma \ref{mex} have endpoints at distance at least 3.  Then these two edges emanating out of $u_i$ partitions $F$ into two parts, $F_s$ and $F_r$, such that $|F_s| = s, |F_r| = r$, $5 \leq s,r \leq k-1$ and $s + r = k  + 4$.  By the pigeonhole principle, $G$ has at least $6(s - 4)$ vertices inside of $F_s$ or $6(r-4)$ vertices inside $F_r$.  Indeed, otherwise the total number of vertices inside of $F$ (including vertex $u_i$) is at most

\[6(s - 4) - 1 + 6(r - 4) - 1 + 1 = 6(k-4) - 1,\]

\noindent contradiction.  Hence, we can apply induction to $F_s$ or $F_r$ to obtain a 4-face inside of $F$ and we are done.
\medskip

\begin{figure} 
     \centering
     \begin{subfigure}[b]{0.25\textwidth}
         \centering
         \includegraphics[width=\textwidth]{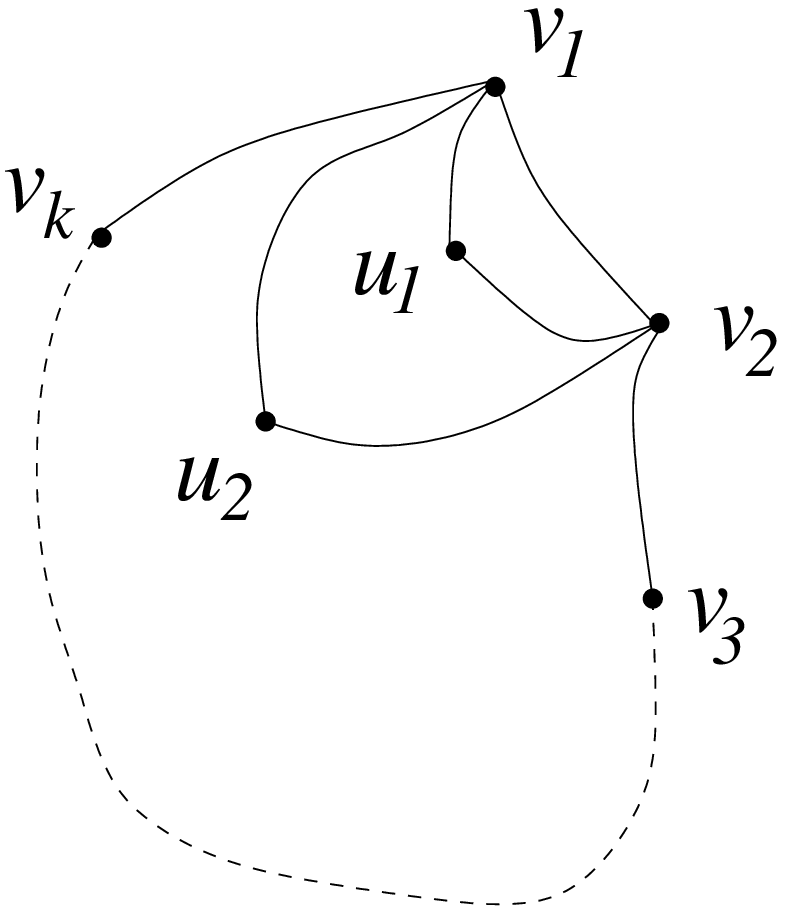}
         \subcaption{Non-crossing edges.}
         \label{fig:insidek1}
     \end{subfigure}
     \hspace{2cm}
     \begin{subfigure}[b]{0.25\textwidth}
         \centering
         \includegraphics[width=\textwidth]{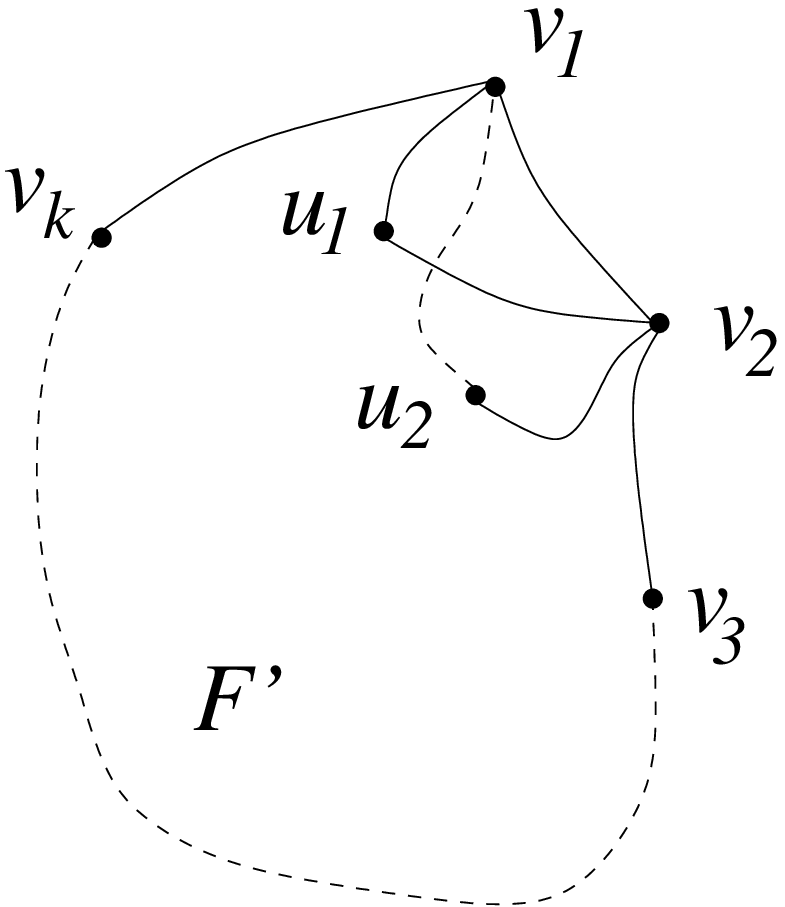}
         \subcaption{Plane subgraph $H'$.}
         \label{fig:insidek2}
     \end{subfigure}
        \caption{Finding a 4-face inside a face of size $k$.}
        \label{fig:insidek12}
\end{figure}

\noindent \emph{Case 3.} Assume for each $u_i$, the two edges emanating out of $u_i$ obtained from Lemma \ref{mex} have endpoints that have distance one along the boundary of $F$ (consecutive vertices along $F$).  Since $t = 6(k-4) > k$, by the pigeonhole principle, there are two vertices, say $u_1$ and $u_2$, such that the two edges emanating out of $u_1$ and $u_2$ obtained from Lemma \ref{mex} go to the same two vertices, say $v_1,v_{2}$.  If these four edges are noncrossing, then we have a triangle with a vertex inside.  By Lemma~\ref{triangle}, we obtain a 4-face inside of $F$ and we are done. See Figure \ref{fig:insidek1}.  Therefore, without loss of generality, we can assume that edges $u_1v_2$ an $u_2v_1$ cross.   

Let $H' = H \cup \{u_1v_1,u_1v_2,u_2v_2\}$, which implies that $u_2$ has degree one in $H'$.  Let $F'$ be the face that contains $u_2$ on its boundary.  See Figure \ref{fig:insidek2}.  We apply Lemma \ref{one} to $u_2$ and the plane graph $H'$, to obtain an edge $u_2v_i$ whose interior lies inside of $F'$ and $v_i$ lies on the boundary of $F'$.  If $v_i = v_3$, then we obtain a 4-face inside of $F$ by following the sequence of vertices $(u_2,v_1,v_2,v_3)$ in $G$.   If $v_i = u_1$, then we obtain a 4-face inside of $F$ by following the sequence of vertices $(u_2,u_1,v_1,v_2)$ in $G$.  If $v_i = v_k$, then again, we obtain a 4-face inside of $F$ by following the sequence of vertices $(u_2,v_k,v_1,v_2)$ in $G$.  

Finally, if $v_i \neq v_k,u_1,v_3$, then either $u_2v_2\cup u_2v_i$ or $u_2v_1\cup u_2v_i$ partitions $F$ into two parts, $F_s$ and $F_r$, such that $|F_s| = s$, $|F_r| = r$, where $5 \leq s,r \leq k-1$ and $s +r = k  + 4$.  By following the arguments in Case 2, we can apply induction on $F_s$ or $F_r$ to obtain a 4-face inside of $F$.  This completes the proof.  \end{proof}

\section{Pairwise disjoint 4-faces in simple drawings}\label{maindis}

\begin{proof}[Proof of Theorem \ref{main}]

Let $G = (V, E)$ be a complete $n$-vertex simple topological graph. We can assume that $n \geq 40$ since otherwise the statement is trivial.  Notice that the edges of $G$ divide the plane into several cells (regions), one of which is unbounded. We can assume that there is a vertex $v_0 \in  V$ such that $v_0$ lies on the boundary of the unbounded cell. Indeed, otherwise we can project $G$ onto a sphere, then choose an arbitrary vertex $v_0$ and then project $G$ back to the plane such that $v_0$ lies on the
boundary of the unbounded cell.  Moreover, the new drawing is isomorphic to the original one as topological graphs. 

Consider the topological edges emanating out from $v_0$, and label their endpoints $v_1,\dots , v_{n-1}$ in clockwise order.    For convenience, we write $v_i \prec v_j$ if $i < j$.   Given subsets $U,W \subset \{v_1,\dots, v_{n-1}\}$, we write $U\prec W$ if $u \prec w$ for all $u \in U$ and $w \in W$.  We start by partitioning our vertex set

\[\mathcal{P}: V(G) = V_0\cup V_1 \cup \cdots \cup V_{\lfloor \frac{n-1}{5}\rfloor},\]

\noindent  such that for $j < \lfloor \frac{n-1}{5}\rfloor$, we have \[V_j = \{v_{5j + 1}, v_{5j + 2}, v_{5j + 3}, v_{5j + 4}, v_{5j + 5}\},\]

\noindent and $\left|V_{\lfloor \frac{n-1}{5}\rfloor}\right| < 5$.  Let $H\subset G$ be a planar subgraph of $G$, and let $T,T'$ be two triangles in $H$ that are incident to $v_0$.   We say that $T$ and $T'$ are \emph{adjacent} if $V(T) = \{v_0,v_i,v_j\}$ and $V(T') = \{v_0,v_j,v_k\}$ such that $v_i \prec v_j \prec v_k$, and the edges $v_0v_i,v_0v_j,v_0v_k$ appear consecutively in clockwise order among the edges emanating out of $v_0$ in $H$.  See Figures \ref{fig:c4} and \ref{fig:c2} for an example.

In what follows, we will construct a plane subgraph $H\subset G$ such that, at each step, we use Lemma \ref{mex} to add at least one edge within the vertex set $\{v_1,\ldots, v_{n-1}\}$.  The goal at each step is to add an edge without creating any empty triangles incident to $v_0$.  If we are forced to create such an empty triangle, we then create another triangle incident to $v_0$ that is adjacent to it, so that we obtain a 4-face. We now give the details of this process.

\begin{lemma}\label{plane}

For each $i \in \{0,1,\ldots, \lfloor n/12\rfloor \}$, there is a plane subgraph $H_i \subset G$ such that $V(H_i) = V(G)$ and $H_i$ satisfies the following properties.

\begin{enumerate}
 
    \item $H_i$ has at least $i$ edges within the vertex set $\{v_1,\ldots, v_{n-1}\}$.
    
    \item The number of parts $V_j \in \mathcal{P}$ with the property that each vertex in $V_j$ has degree one in $H_i$ is at least $\lfloor (n-1)/5\rfloor - 2i$.
    
    \item If the vertex set $\{v_0,v_k,v_{\ell}\}$ induces an empty triangle $T$ in $H_i$, then both vertices $v_k,v_{\ell}$ must lie in the same part $V_j \in \mathcal{P}$ and $\ell = k + 1$.  Moreover, given that such an empty triangle $T$ exists, there must be another triangle $T'$ adjacent to $T$ in $H_i$, such that $V(T') = \{v_0,v_t,v_{t'}\}$ and $v_t,v_{t'} \in V_j$.
    
\item If the edge $v_0v_t$ is not in $H_i$, then $v_t$ is an isolated vertex in $H_i$.
\end{enumerate}

\end{lemma}

 \begin{proof}

We start by setting $H_0$ as the plane subgraph of $G$ consisting of all edges emanating out of $v_0$.  Clearly, $H_0$ satisfies the properties above.  For $i < n/12$, having obtained $H_i$ with the properties described above, we obtain $H_{i + 1}$ has follows.  

Fix a part $V_j  \in \mathcal{P}$ such that each vertex in $V_j$ has degree one in $H_i$ and $|V_j| = 5$.   For simplicity, set $u_i = v_{5j + i}$, for $i \in \{1,\ldots, 5\}$, which implies $V_j = \{u_1,u_2,u_3,u_4,u_5\}$.  Since

\[\lfloor \frac{n-1}{5}\rfloor - 2i \geq \frac{n-1}{5} - \frac{n}{6} > 1,\]

\noindent such a part $V_j \in \mathcal{P}$ exists.  Clearly, all vertices in $V_j$ lie on the boundary of a face $F$ in the plane graph $H_i$. We then apply Lemma \ref{one} to the plane graph $H_i$ and the vertex $u_3$, and obtain edge $u_3v_k$, whose interior lies within $F$ and $v_k$ is on the boundary of $F$.  We now consider the following cases.

\medskip

\noindent \emph{Case 1.}  Suppose  $v_k \neq u_2, u_4$.  See Figure \ref{fig:vk1}.  We then set $H_{i + 1} = H_i\cup \{u_3v_k\}$.  Clearly, $H_{i + 1}$ is planar. Moreover, the number of edges in $H_{i + 1}$ within the vertex set $\{v_1,\ldots, v_{n -1}\}$ is at least $i + 1$.   Also, the only vertices that no longer have degree one in $H_{i + 1}$ are $u_3$ and $v_k$.  Hence, the number of parts $V_{\ell} \in \mathcal{P}$ with the property that all vertices in $V_{\ell}$ have degree one in $H_{i + 1}$ is at least \[\lfloor \frac{n - 5}{2}\rfloor - 2i - 2 = \lfloor \frac{n - 5}{2}\rfloor  - 2(i + 1).\]  Since $v_k\neq u_2, u_4$, no empty triangles incident to $v_0$ were created.  Also, no edge emanating out of $v_0$ was deleted from $H_i$.  Thus, $H_{i + 1}$ satisfies the conditions described above.

\begin{figure}[ht]
     \centering
     \begin{subfigure}[b]{0.35\textwidth}
         \centering
         \includegraphics[width=\textwidth]{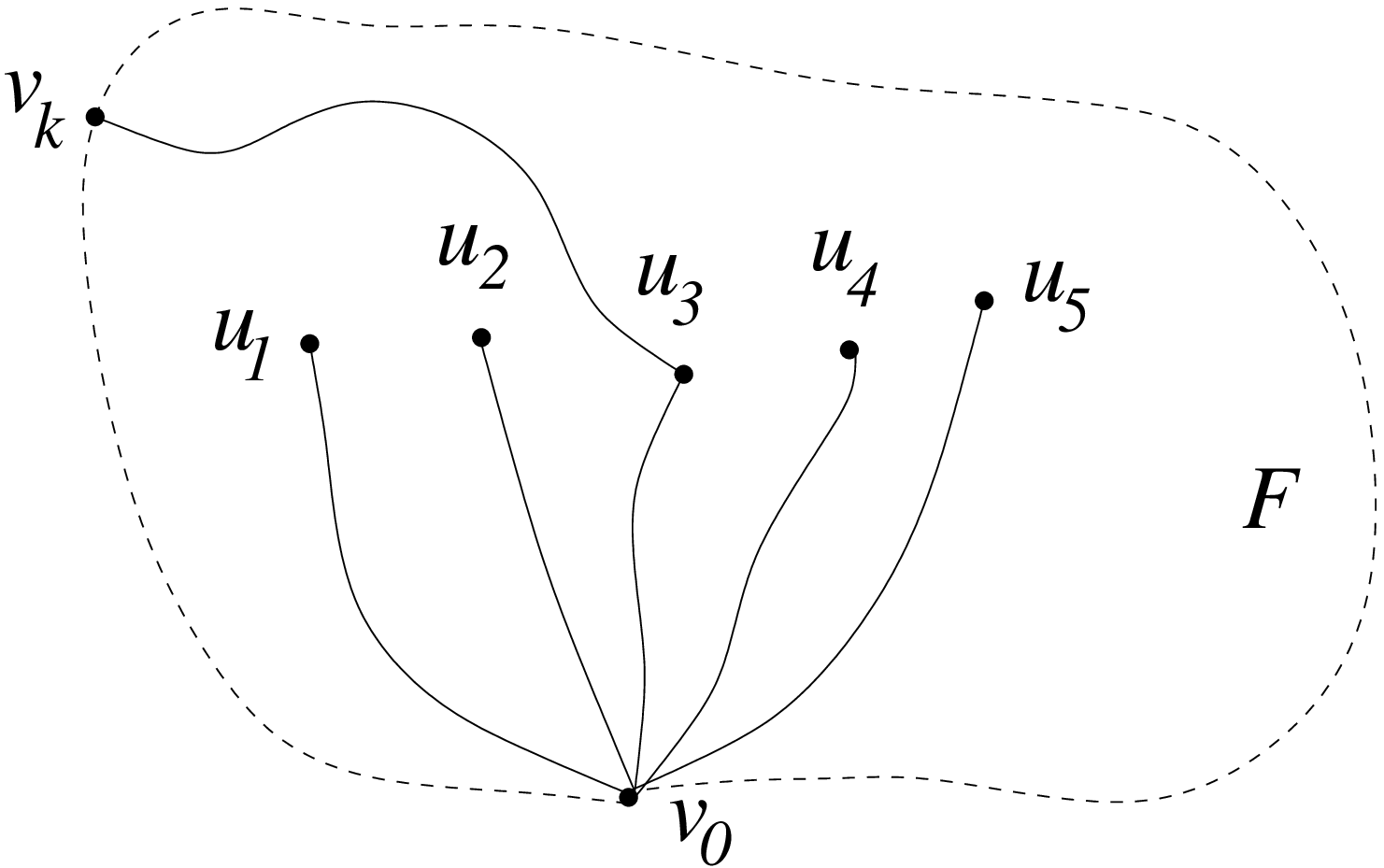}
         \subcaption{$v_k \neq u_2,u_4$.}
         \label{fig:vk1}
     \end{subfigure}
     \hspace{2cm}
     \begin{subfigure}[b]{0.35\textwidth}
         \centering
         \includegraphics[width=\textwidth]{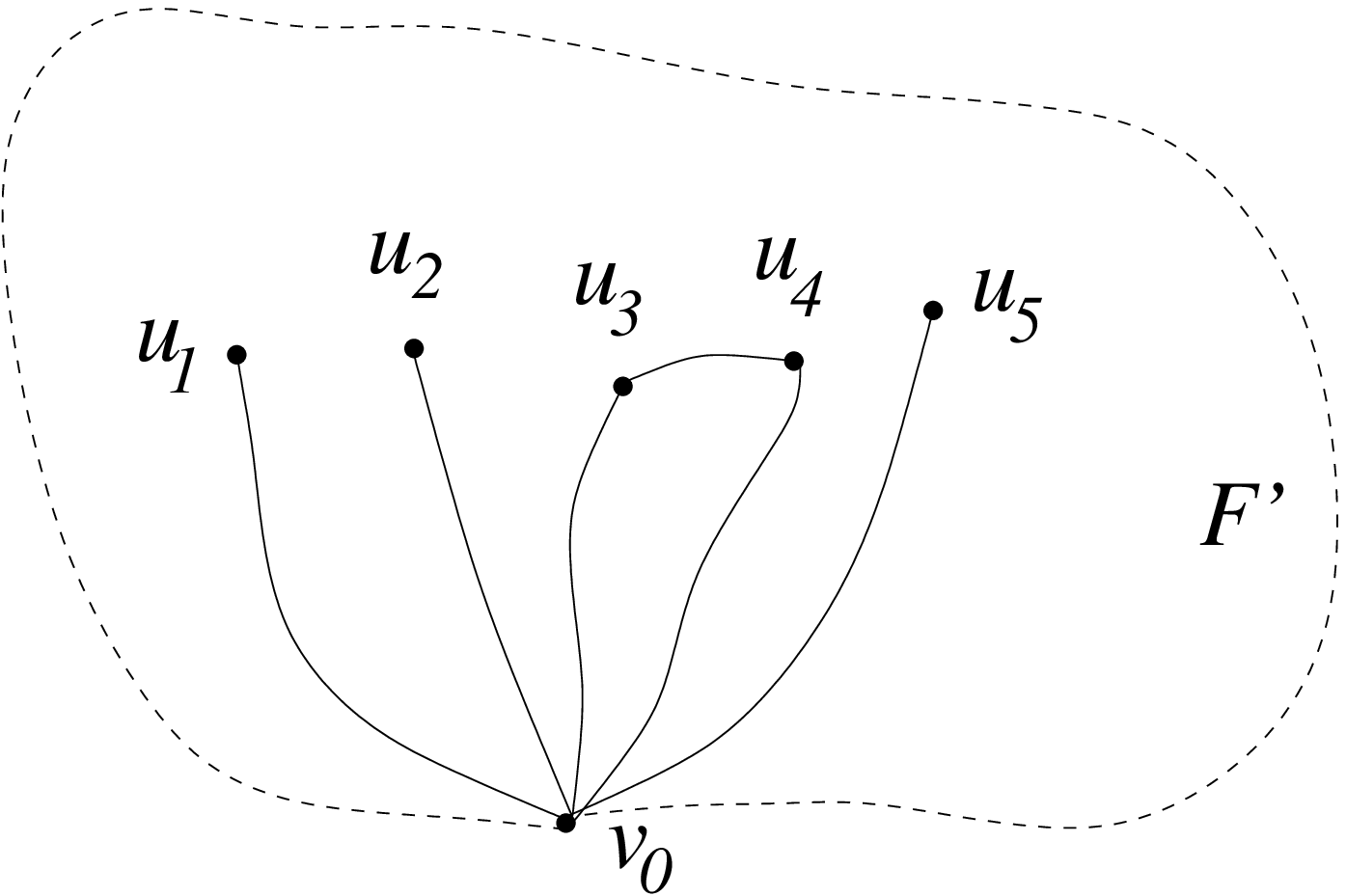}
         \subcaption{$v_k = u_4$.}
         \label{fig:vk2}
     \end{subfigure}
        \caption{Cases 1 and 2 in Lemma \ref{plane}.}
        \label{fig:vkad}
\end{figure}

\medskip

\noindent \emph{Case 2.}  Suppose $v_k= u_2$ or $v_k= u_4$.    Without loss of generality, we can assume $v_k= u_4$, since otherwise a symmetric argument would follow.  If there is a vertex of $G$ inside the triangle $T = \{v_0,u_3,u_4\}$, then we set $H_{i + 1} = H_i\cup\{u_3u_4\}$. By the same arguments as above, $H_{i + 1}$ satisfies the properties described above.

Hence, we can assume that the triangle $T$, where $V(T)= \{v_0,u_3,u_4\}$, is empty in $G$.  Set $H' = H_i\cup \{u_3u_4\}$, and let $F'$ be the face in $H'$ whose boundary contains $u_2$.  See Figure \ref{fig:vk2}.   We apply Lemma \ref{one} to $H'$ and $u_2$ and obtain another edge $u_2v_{\ell}$ whose interior lies inside $F'$.  If $v_{\ell} = u_3$, then we set $H_{i + 1} = H_i\cup \{u_3u_4, u_2u_3\}$, which implies that the empty triangle $T$ is adjacent to triangle $T'$, where $V(T')= \{v_0,u_2,u_3\}$.    By a similar argument as above, $H_{i + 1}$ satisfies the properties above.  If $v_{\ell} \neq u_1,u_3$, then edge $u_2v_{\ell}$ does not create any empty triangles incident to $v_0$ and we set $H_{i + 1} = H_{i}\cup \{u_2v_{\ell}\}$.   By the same argument as above, $H_{i + 1}$ satisfies the desired properties.

Finally, let us consider the case that $v_{\ell} = u_1$.  If the triangle $T'$ is not empty, where $V(T') = \{v_0,u_1,u_2\}$, we set $H_{i + 1} = H_i\cup \{u_1u_2\}$ and we are done by the arguments above.  Therefore, we can assume that the triangle $T'$ is also empty. 

Let $H'' = (H_i\cup \{u_1u_2\})\setminus \{u_3\}$.  Let $F''$ be the face whose boundary contains $u_4$ in $H''$.   See Figure \ref{fig:ad0}.  We apply Lemma \ref{one} to $H''$ and the vertex $u_4$ to obtain edge $u_4v_t$ whose interior lies inside $F''$.  We now examine $H_i\cup\{u_1u_2,u_3u_4,u_4v_t\}$. The proof now falls into the following cases.

\begin{figure}[ht]
     \centering
          \begin{subfigure}[b]{0.3\textwidth}
         \centering
         \includegraphics[width=\textwidth]{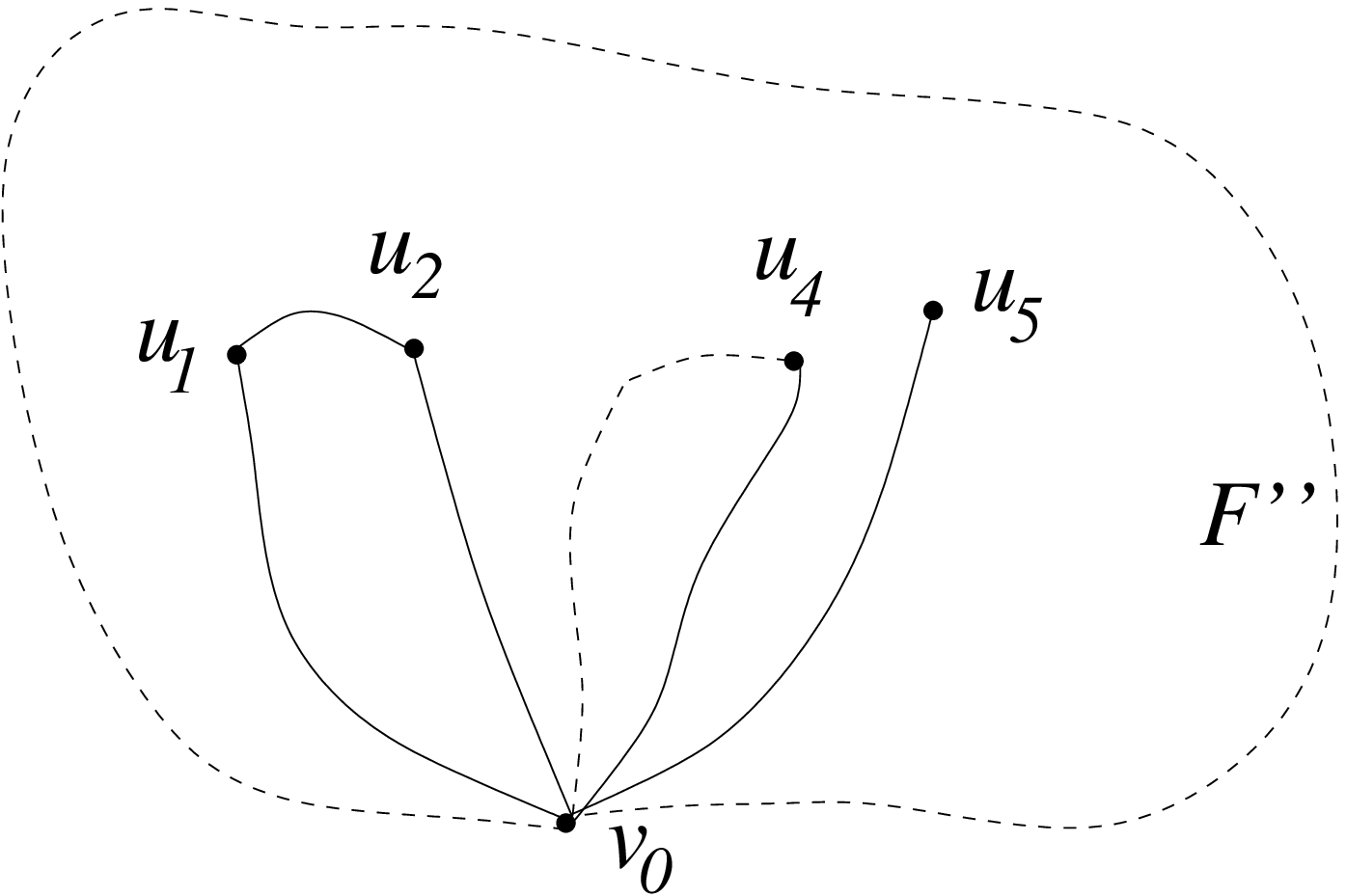}
         \subcaption{$v_{\ell} = u_1$, the plane graph $H''$.}
         \label{fig:ad0}
     \end{subfigure}
     \hspace{1.5cm}
     \begin{subfigure}[b]{0.3\textwidth}
         \centering
         \includegraphics[width=\textwidth]{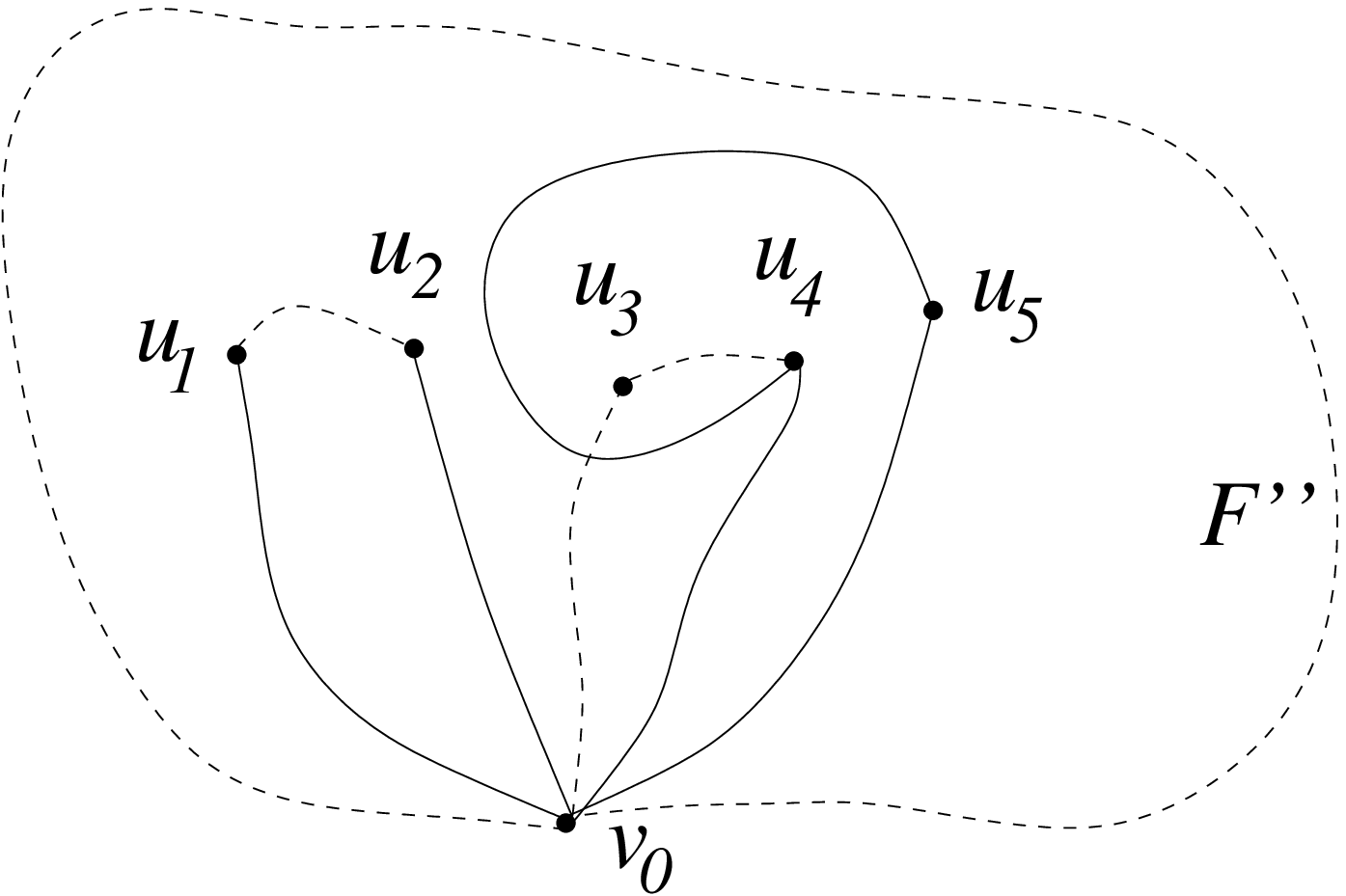}
         \subcaption{$v_t = u_5$, $v_0v_3$ crosses $u_4u_5$.}
         \label{fig:c1}
     \end{subfigure}
     \hspace{.5cm}
     \begin{subfigure}[b]{0.3\textwidth}
         \centering
         \includegraphics[width=\textwidth]{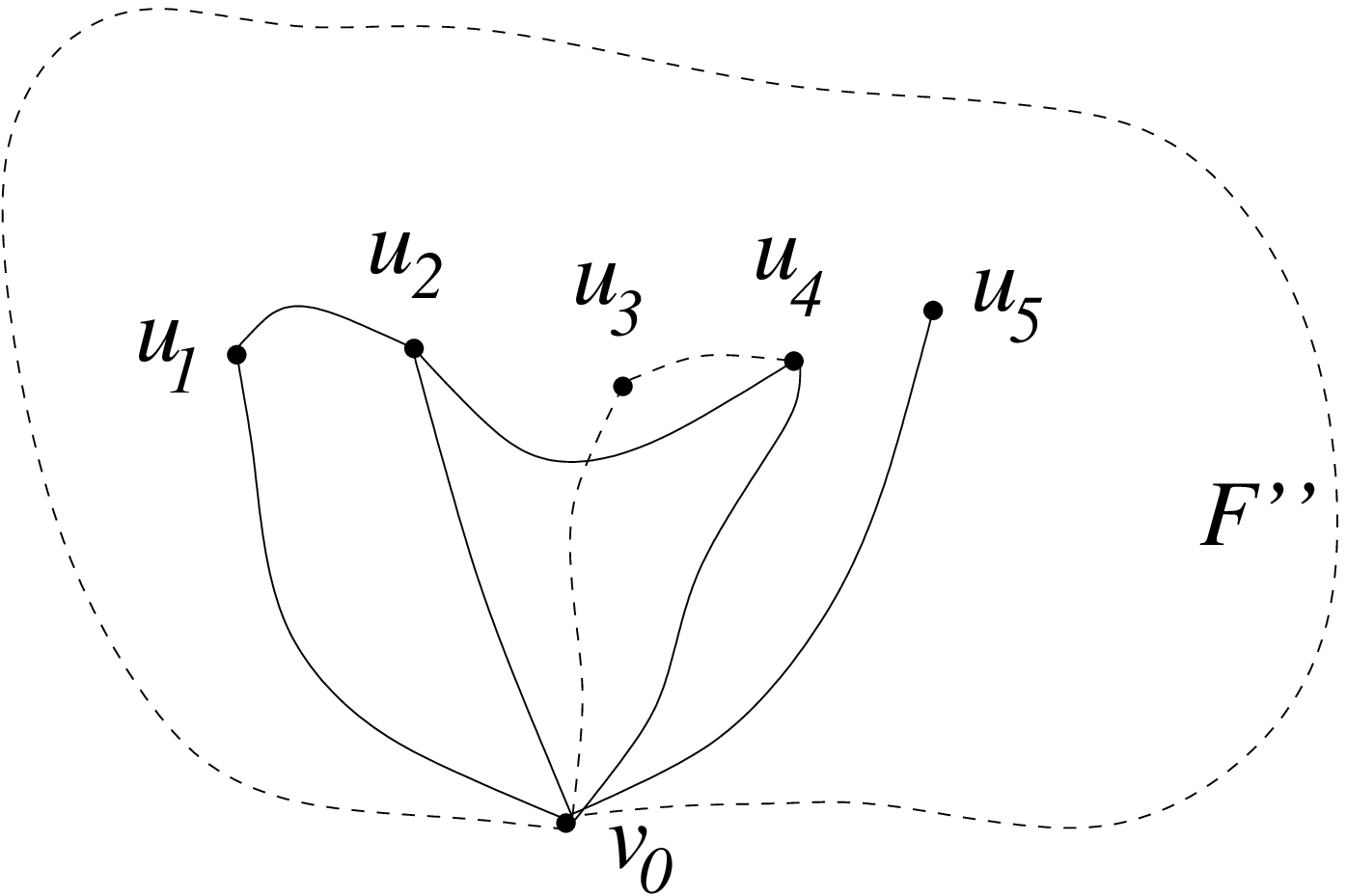}
         \subcaption{$v_t = u_2$}
         \label{fig:c4}
     \end{subfigure}
      \hspace{1.5cm}
           \begin{subfigure}[b]{0.3\textwidth}
         \centering
         \includegraphics[width=\textwidth]{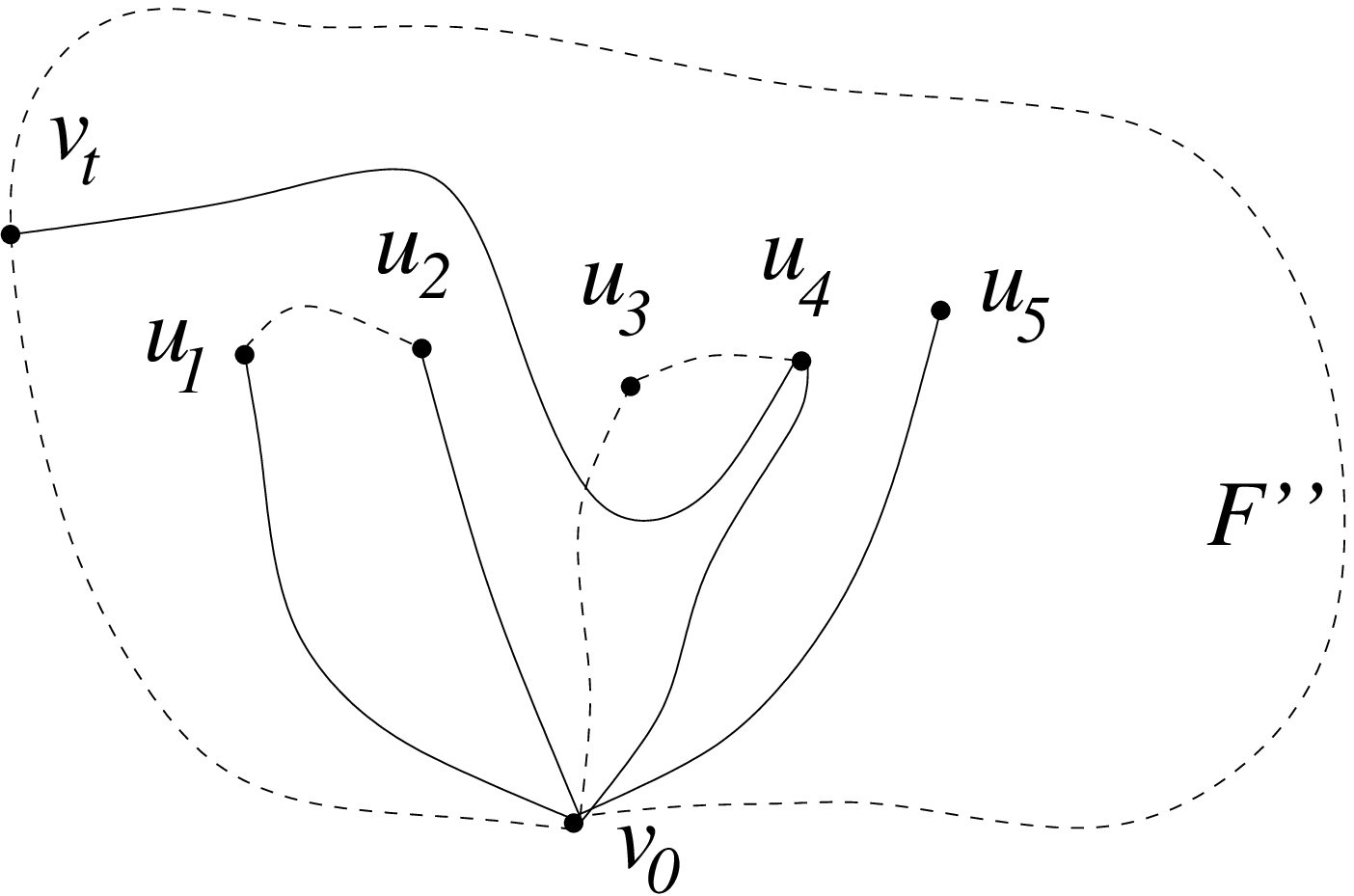}
         \subcaption{$v_t \neq u_2,u_5$}
         \label{fig:c6}
     \end{subfigure}
        \caption{Cases 2.a in Lemma \ref{plane}.  Edge $u_4v_t$ crosses $v_0v_3$.}
        \label{fig:cross}
\end{figure}

\noindent Case 2.a.  Suppose edge $u_4v_t$ crosses edge $v_0u_3$.  If $v_t = u_5$, then $\{v_0,u_4,u_5\}$ induces a non-empty triangle in $G$, so we set $H_{i + 1} = H_i\cup \{u_4u_5\}\setminus \{v_0v_3\}$. Then $u_3$ is an isolated vertex in $H_{i + 1}$ and we did not create any empty triangles incident to $v_0$, and we are done.  See Figure \ref{fig:c1}.  If $v_t = u_2$, then we set $H_{i + 1} = H_i\cup \{u_1u_2,u_2u_4\} \setminus \{v_0u_3\}$.  Then the empty triangle on $\{v_0,u_1,u_2\}$ is adjacent to the triangle on $\{v_0,u_2,u_4\}$ in $H_{i + 1}$, $u_3$ is an isolated vertex, and we are done. See Figure \ref{fig:c4}. Otherwise, if $v_t\neq u_2,u_5$, we set $H_{i + 1} = (H_i \cup \{u_4v_t\})\setminus \{v_0u_3\}$.  Then $u_3$ is an isolated vertex, we do not create any empty triangles incident to $v_0$, and we are done. See Figure \ref{fig:c6}.

\medskip

\noindent Case 2.b.  Suppose edges $v_0u_3$ and $u_4v_t$ do not cross.  If $v_t = u_5$, then we set $H_{i + 1} = H_i \cup \{u_3u_4,u_4u_5\}$, and the empty triangle on the vertex set $\{v_0,u_3,u_4\}$ is adjacent to the triangle on $\{v_0,u_4,u_5\}$, and we are done.  See Figure \ref{fig:c2}. If $v_t \neq u_5$, then we set $H_{i + 1} = H_i\cup \{u_4v_t\}$.  Since we do not create any empty triangles incident to $v_0$, we are done.    See Figure \ref{fig:c3}.  This completes the proof of the statement.
 \end{proof}

\begin{figure}[ht]
     \centering
     \begin{subfigure}[b]{0.35\textwidth}
         \centering
         \includegraphics[width=\textwidth]{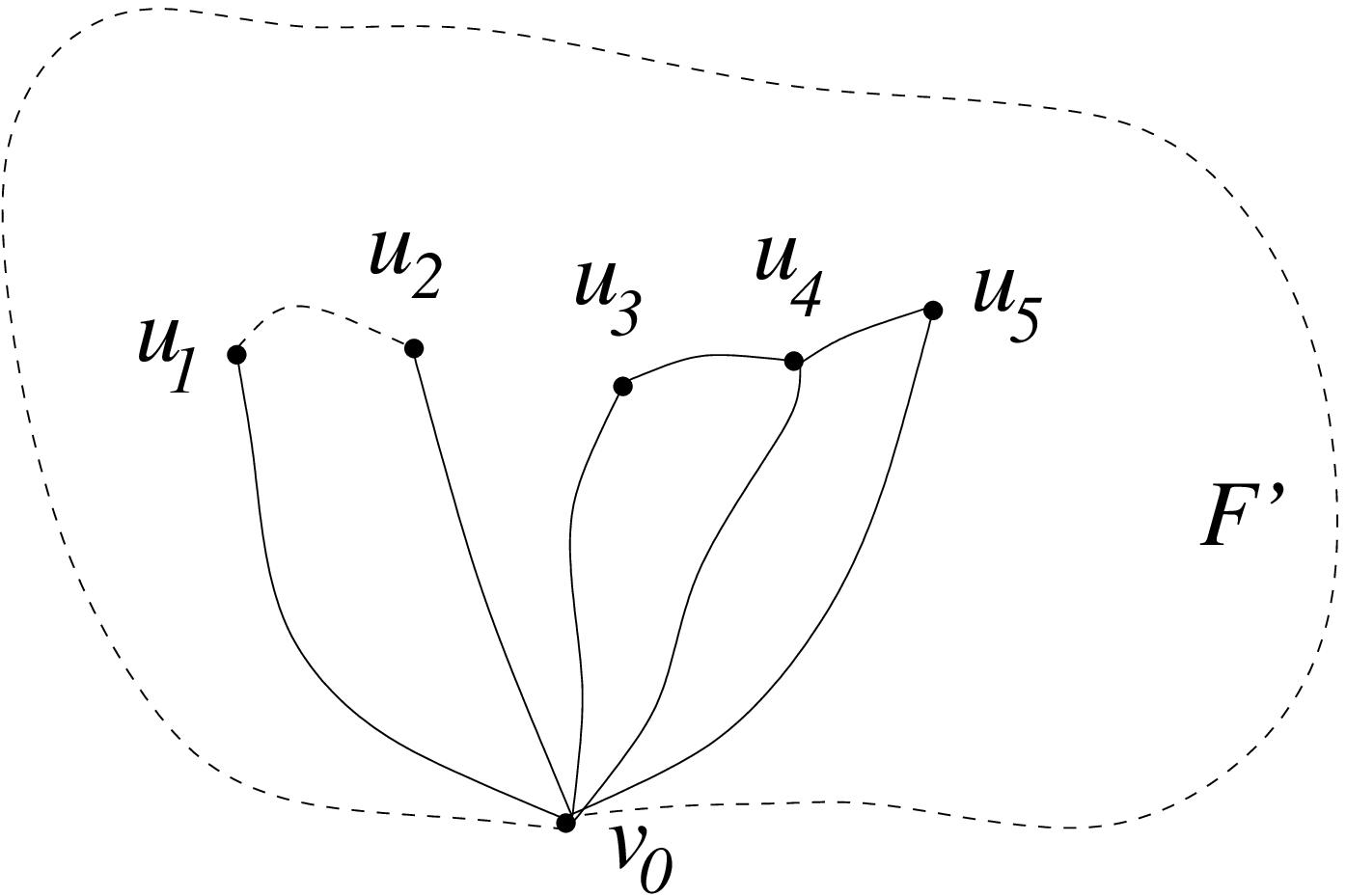}
         \subcaption{$v_t = u_5$}
         \label{fig:c2}
     \end{subfigure}
     \hspace{1cm}
     \begin{subfigure}[b]{0.35\textwidth}
         \centering
         \includegraphics[width=\textwidth]{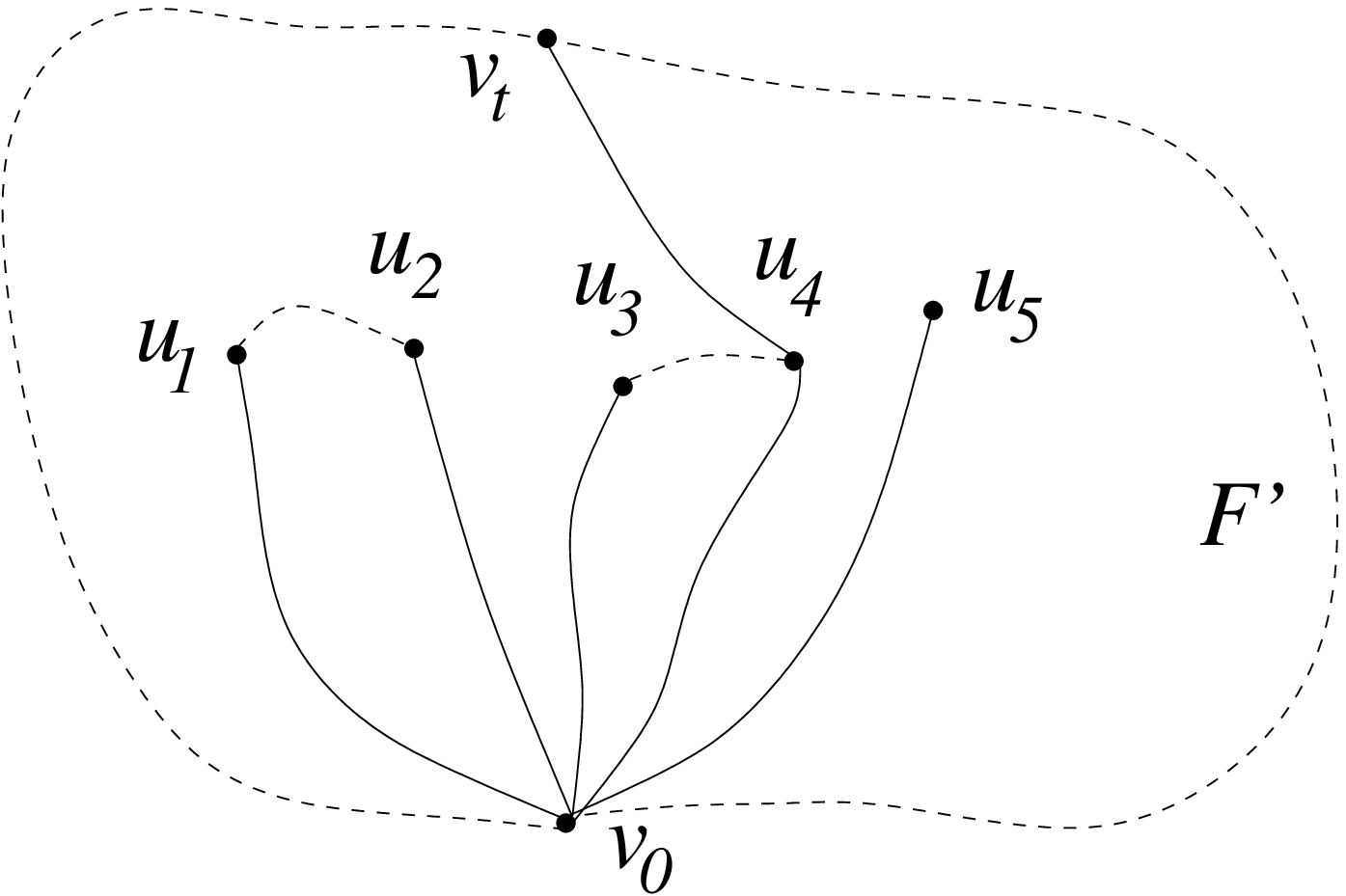}
         \subcaption{$v_t \neq u_5$}
         \label{fig:c3}
     \end{subfigure}
        \caption{Case 2.b in Lemma \ref{plane}.  Edge $u_4v_t$ does not cross $v_0v_3$.}
        \label{fig:disjoint2}
\end{figure}

\medskip

Set $H = H_{\lfloor n/12\rfloor }$.  We now will use the planar graph $H$ and the vertices of $G$ to find many pairwise disjoint 4-faces. If there is a vertex $v_j \in \{v_1,\ldots, v_{n-1}\}$ with degree at least $n^{1/3}$ in $H$, then together with $v_0$, we have a planar $K_{2,\lfloor n^{1/3}\rfloor}$.  This gives rise to $\Omega(n^{1/3})$ pairwise disjoint 4-faces and we are done.

Hence, we can assume that every vertex $v_i \in \{v_1,\ldots, v_{n-1}\}$ has degree at most $n^{1/3}$.  Since there are at least $n/12$ edges induced on the vertex set $\{v_1,\ldots, v_{n-1}\}$ in the plane graph $H$, there is a plane matching $M$ on  $\{v_1,\ldots, v_{n-1}\}$ of size at least $n^{2/3}/16$.   Notice the there is a natural partial ordering $\prec^{\ast}$ on $M$.   Given two edges $v_iv_j, v_kv_{\ell} \in M$, we write $v_kv_{\ell} \prec^{\ast} v_iv_{j}$ if $v_i\prec v_k\prec v_{\ell}\prec v_j$.  By Dilworth's theorem, $M$ contains either a chain or antichain of length at least $n^{1/3}/4$ with respect to the partial ordering $\prec^{\ast}$.  The proof now falls into two cases.

\medskip

\noindent \emph{Case 1.}  Suppose we have an antichain $M'$ of size $n^{1/3}/4$.   Let \[M' = \{v_{\ell_1}v_{r_1}, v_{\ell_2}v_{r_2}, \ldots,v_{\ell_t}v_{r_t} \},\] where $t = n^{1/3}/4$  and $\ell_i < r_i$ for all $i$.  Since $H$ is planar, we have

\[\{v_{\ell_1} , v_{r_1}\} \prec \{v_{\ell_2} , v_{r_2}\} \prec \cdots \prec  \{v_{\ell_{t}} , v_{r_{t}}\}.\]

\noindent  See Figure \ref{fig:ant1}.  If at least half of the edges in $M'$ give rise to a non-empty triangle incident to $v_0$, then we apply Lemma~\ref{triangle} to each such triangle to obtain $\Omega(n^{1/3})$ pairwise disjoint 4-faces.  Hence, we can assume at least half of these triangles are empty.  By construction of $H$, each such empty triangle has another triangle adjacent to it.  Since the three edges emanating out of $v_0$ of two adjacent triangles must be consecutive in $H$ (by definition), this corresponds to $\Omega(n^{1/3})$ pairwise disjoint 4-faces.  See Figure \ref{fig:ant2}.

\begin{figure}[ht]
     \centering
     \begin{subfigure}[b]{0.3\textwidth}
         \centering
         \includegraphics[width=\textwidth]{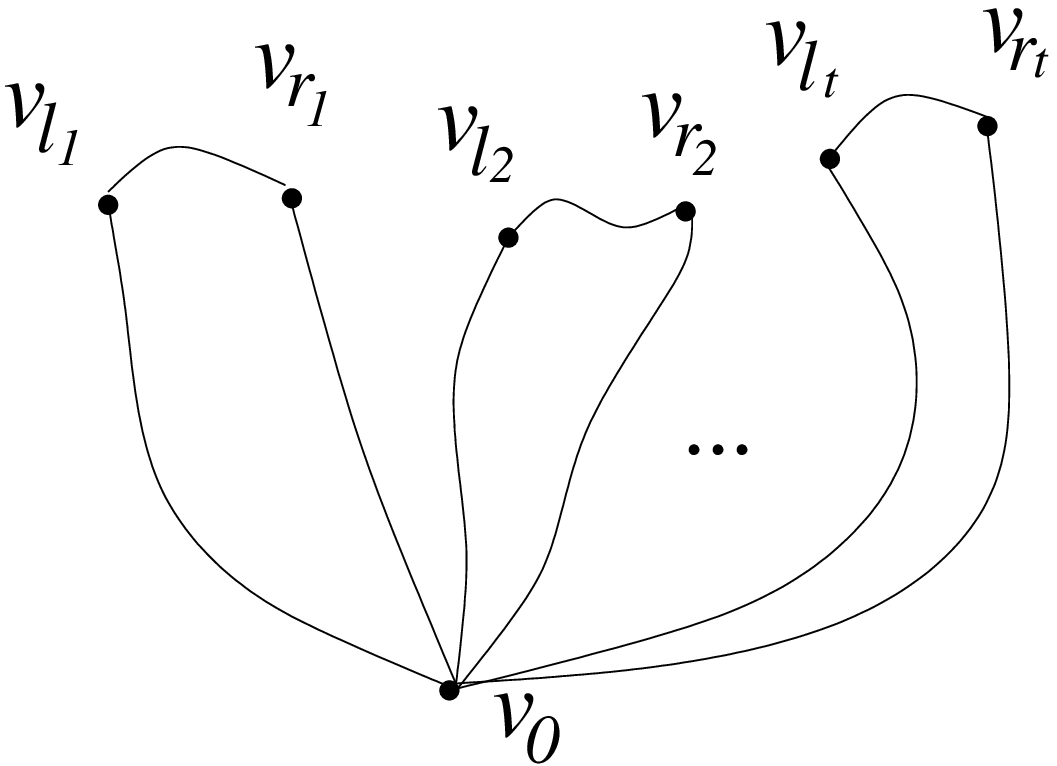}
         \subcaption{Anti-chain of size $t$.}
         \label{fig:ant1}
     \end{subfigure}
     \hspace{.25cm}
     \begin{subfigure}[b]{0.3\textwidth}
         \centering
         \includegraphics[width=\textwidth]{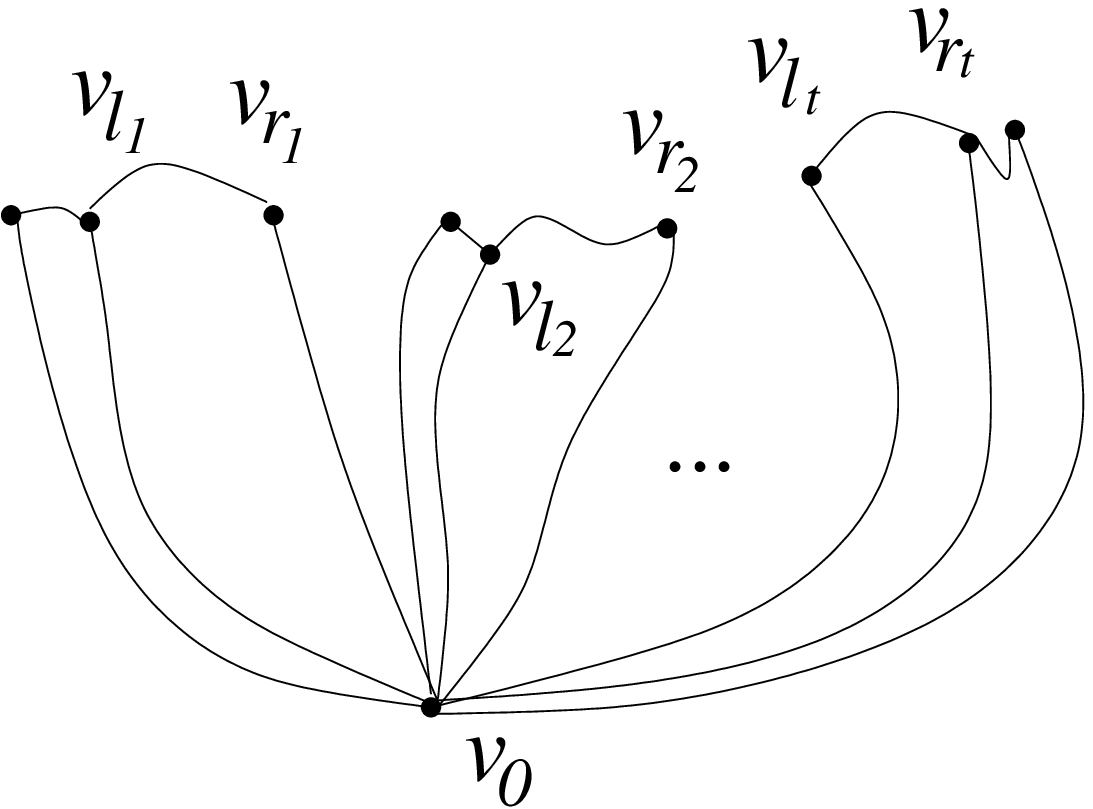}
         \subcaption{Disjoint 4-faces.}
         \label{fig:ant2}
     \end{subfigure}
          \hspace{.25cm}
         \begin{subfigure}[b]{0.3\textwidth}
         \centering
         \includegraphics[width=\textwidth]{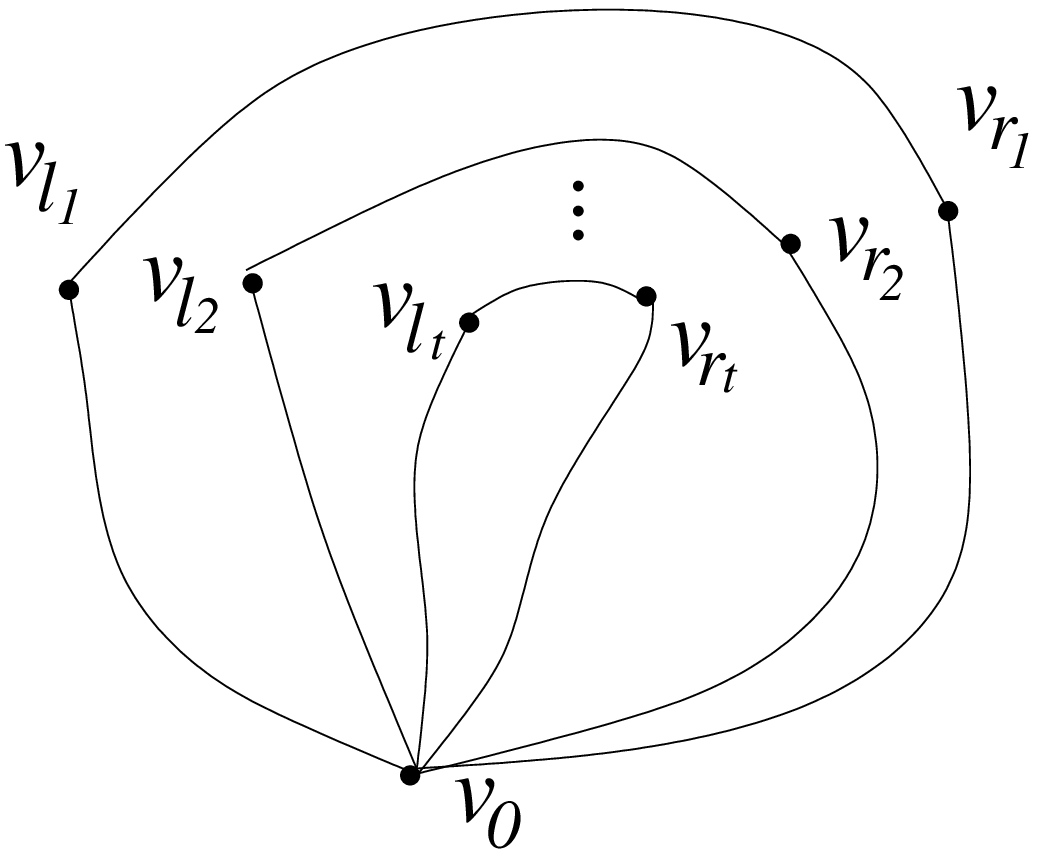}
         \subcaption{Chain of length $t$}
         \label{fig:chain}
     \end{subfigure}
        \caption{Large antichain and chain.}
        \label{fig:disjoint}
\end{figure}

\medskip

\noindent Case 2.  Suppose we have a chain $M'\subset M$ of size $n^{1/3}/4$. Hence, 
 
\[M' = \{v_{\ell_1}v_{r_1}, v_{\ell_2}v_{r_2}, \ldots,v_{\ell_t}v_{r_t} \},\] where $t = n^{1/3}/4$  and we have

\[v_{\ell_{t}}v_{r_{t}} \prec^{\ast}   \cdots\prec^{\ast} v_{\ell_2}v_{r_2} \prec^{\ast} v_{\ell_1}v_{r_1}.\]

 \noindent See Figure \ref{fig:chain}.  Set $M'' \subset M'$ such that $M'' = \{v_{\ell_{7j}}v_{r_{7j}}\}_j$.  Hence, $|M''| \geq \Omega(n^{1/3})$.  Let us consider edges $v_{\ell_7}v_{r_7}$ and  $v_{\ell_{14}}v_{r_{14}}$ from $M''$, and the region $F$ enclosed by the six edges.

\[v_{\ell_7}v_{r_7}, v_{\ell_{14}}v_{r_{14}}, v_0v_{\ell_7}, v_0v_{\ell_{14}},v_0v_{r_7}, v_0v_{r_{14}}.\]

\noindent   See Figure \ref{fin}.  Let $H'$ be the plane graph on the vertex set $\{v_0,v_{\ell_{14}},v_{r_{14}},v_{r_7},v_{\ell_7}\}$ and the six edges listed above.   By construction of $M''$, we know that there are at least 12 vertices of $V(G)$ inside $F$.   Since $|F| = 6$, we can apply Lemma \ref{key} to find a $4$-face inside of $F$.  By repeating this argument for each consecutive pair of edges in the matching $M''$ with respect to the partial order $\prec^{\ast}$, we obtain $\Omega(n^{1/3})$ pairwise disjoint 4-faces.   \end{proof}

\begin{figure}[h]
\centering
\includegraphics[width=4cm]{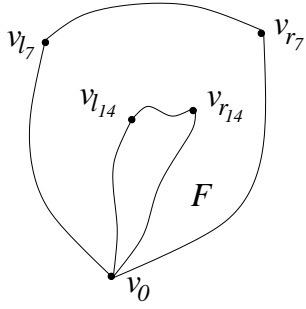}
\caption{Face $F$ of size 6.}\label{fin}
\end{figure}

\section{$\mathbb{Z}_2$-cycles in topological graphs}\label{last}

Now we pass to a variant of Heilbronn's triangle problem for not necessarily simple topological graphs.  Specifically, if $\gamma$ is piece-wise smooth closed curve with transverse self intersections, then one can consider Claim~\ref{parity} as a definition of the $\mathbb Z_2$-inside of $\gamma$.  That is, $p$ is in the interior of $\gamma$ if any arc with one endpoint at $p$ and the other outside a large disk containing $\gamma$ intersects $\gamma$ an odd number of times at proper crossings.

Does every complete topological graph drawn inside the unit square contain a cycle whose $\mathbb Z_2$-inside has area $o(1)$?  More generally, one can ask this question for the group of $\mathbb Z_2$-cycles of the complete graph instead of graph cycles as we explain below.

Before going into the technical details, let us try to give some context.  For $p$ prime and $p+2=3k$, the topological Tverberg theorem \cite{Tverberg} implies that for any drawing of the complete graph $K_{p+2}$, the vertex set can be partitioned into $k$ triples of vertices such that the corresponding triangles intersect in the plane.  

Looking for disjoint faces, as in Theorem \ref{main}, can be viewed as the dual to the topological Tverberg theorem. One looks for disjoint triangles instead of intersecting triangles, but evidently, one needs some conditions on the drawing for the conclusion of Theorem \ref{main} to hold.  In the face-wise linear case, Tverberg's theorem implies B\`ar\`any's overlap Theorem \cite{barany} (also known as the first selection lemma). Gromov \cite{gromov} showed that this theorem holds for continuous maps. In particular, in any drawing of the complete graph, there exists a point in $\mathbb R^2$ contained in at least $\frac{2}{9}$ of all the triangles. This is somehow dual to finding a cycle with small area. 

\subsubsection{Chain complexes}
Let us recall the basic objects of cellular homology, refer to \cite{homology} for a gentle introduction.  
If $X$ is a cell complex for each $i$ let $C_i(X,\mathbb Z_2)$ be the group of formal linear combinations of the $i$-dimensional cells. An element of $C_i(X,\mathbb Z_2)$ has the form $\sum_{\sigma \in F_i(X)} a_\sigma \sigma$, 
where $\sigma$ is an element of $F_i$, the set of $i$-dimensional faces, and $a_\sigma$ is an element of $\mathbb Z_2$ the field with two elements. In what follows, we will consider $K_n$ as a simplicial complex, in other words, $F_1(K_n)$ is the set of edges and $F_0(K_n)$ is the set of vertices of the complete graph.  

The boundary  $\partial \colon C_1(K_n,\mathbb Z_2) \to C_0(K_n,\mathbb Z_2)$ is a linear map between these abelian groups defined as follows: if $e=(i,j)$ is an edge, then the chain $1 e$ is mapped to $\partial(e)=1i+1j$. The kernel of $\partial$ is the group of $1$-cycles of $K_n$, $Z_1(K_n):=\ker \partial$.

Notice that elements in $Z_1$ can be identified with (possibly disjoint) graphs in which every vertex has even degree.

If $G$ is a complete topological graph, it induces a cell decomposition of a topological disk in the plane.  Let us consider the plane graph induced by $G$ by introducing a vertex at every intersection and let $\hat{G}$ be the cell decomposition of the smallest closed topological disk that contains $G$ in which every intersection is a vertex, consecutive vertices share an edge and the regions of $\mathbb R^2 \setminus G$ are the $2$ dimensional cells. We denote this cell decomposition by $\hat{G}$. We consider the chain groups $C_i(\hat{G},\mathbb Z_2)$, and observe that for $i=0,1$ there exists chain maps $f_i\colon C_i(K_n,\mathbb Z_2) \to C_i(\hat{G},\mathbb Z_2)$. For example, for a given edge $e \in E(K_n)$, $f_1(e)$ is the linear combination of the edges in $G$ that support the arc representing $e$, and similarly, for the vertices.

It is not hard to see that this chain map induces a map $f_1\colon Z_1(K_n) \to Z_1(\hat{G})$ between cycle groups. Now, since the homology group $H_1(\hat{G},\mathbb Z_2)$ is trivial, for any cycle $z \in Z_1(\hat{G})$ there exists a $2$ chain $c \in C_2(\hat{G},\mathbb Z_2)$ such that $\partial(c)=z$.
On the other hand, if some other chain $c'$ satisfies $\partial c'=z$, then $\partial(c+c')=0$, hence $c+c'$ would be a two dimensional cycle, and since there are no $3$ dimensional faces, this would imply that the homology group $H_2(\hat{G},\mathbb Z_2) \neq 0$, which is absurd. So the chain $c$ such that $\partial c=z$ is uniquely defined and its interior corresponds to the set of points $\{p\in \mathbb R^2: lk_2(p,z)=1\}$. 

\subsection{A topological graph without $\mathbb Z_2$-cycles of small area.}

\begin{proposition}\label{mod2} There exists a drawing of the complete graph inside $[0,1]^2$ such that the $\mathbb Z_2$-inside of every $\mathbb Z_2$-cycle of the complete graph has area at least $\frac{1}{4}$.\end{proposition}

We begin describing the random construction. Consider a rectangle of size $(m+1) \times 1$, with corners at $\{(0,0),(0,1),(m,0),(m,1)$ where $m$ will be a very large number with respect to $n$ that we will define later on. We perform the area analysis for this drawing, but notice that by applying the linear transformation $(x,y)\to (\frac{x}{m},y)$, we can transform it back to the unit square.

We place all the points in general position on a small neighbourhood of the lower corner (0,0) of the rectangle. The drawing will be random and at the end it will be perturbed by a arbitrary small amount so that it is in general position. To refer to this small perturbation we use the word near in the description below. Notice that one could perturb each edge so that it stays piece-wise linear or one could smooth each edge as long as areas of cycles do not change too much and every intersection is a proper crossing (in the language of differentiable topology this corresponds to the curves being transverse and in PL topology to general position).

Each edge will go all the way to near $(m,0)$ and come back near $(0,0)$. Choose two vertices $i,j$, the edge $e=(i,j)$ will be represented by an arc that begins at the vertex $i$ and is a concatenation of almost vertical and almost horizontal arcs. More specifically, for each $k \in \{1,2,\ldots m-1\}$ assume that we have constructed a path $\alpha_{ij}(k)$ that begins at $i$ (near $(0,0)$) and ends at $(k,y_k)$ with $y_k \in \{0,1\}$, let $Y_{k+1}$ be a Bernoulli random variable with probability $\frac{1}{2}$, and extend the arc $\alpha_{ij}(k)$ by concatenating it with the segment  $\{(k+t,y_k): t \in [0,1]\}$ if $Y_{k+1}=y_k$, and by the concatenation of the segments $\{(k,t): t \in [0,1]\}$ followed by $\{(k+t,Y_{k+1}): t \in [0,1]\}$ if $Y_{k+1}\neq y_k$.

When we reach $x=m$, if $y=1$, we concatenate down to $(m,0)$. In both cases $y=0,1$, we end the arc by concatenating all the way back to the vertex $j$ near $(0,0)$ with a long near horizontal arc close to the $x$-axis. Finally, we perturb what we have constructed a very small amount so that the intersections between any two such edges is a finite set of points where they cross properly, and we re-scale the $x$-axis so that the whole picture is contained in the unit square.

\begin{figure}\label{2cycles}
\centering
\begin{tikzpicture}

	\draw  plot  coordinates{(0,0)(2 ,0 )(2,1)(3,1)(3,0)(4,0)(6,0)(6,1)(8,1)(8,-1/10)(1/10,-1/10)(1/10,2/10)};
	\filldraw [black] (0,0) circle [radius=1pt];
	\filldraw [black] (1/10,2/10) circle [radius=1pt];

\end{tikzpicture}

 \caption{A possible edge in the construction of proposition \ref{mod2} before re-scaling and perturbing.}
\end{figure}
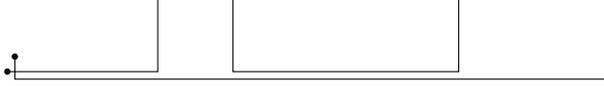

\begin{proof}[Proof of Proposition \ref{mod2}]
We work with the rectangle and make some observations about the re-scaling and perturbing at the end of the proof.  Using the random construction described above, to compute the expected area of a cycle $z$, consider a point $p$ in the interior of the rectangle, say that $p$ has coordinates $(k+\frac{1}{2},\frac{1}{2})$ and one of the segments of the $1$-cycle that joins $(k,Y_k)$ with $(k+1,Y_{k+1})$. The vertical ray emanating up from $p$, denoted by $\alpha$, intersects the edge with probability $\frac{1}{2}$. Hence, conditioned on all the edges of $z$ but one, the square $\{(k,0),(k,1),(k+1,1),(k+1,0)\}$ is completely inside $z$ with probability $\frac{1}{2}$ and completely outside $z$ with probability $\frac{1}{2}$. This implies that the area of every cycle is (up to an arbitrarily small error term from the perturbation) a sum of $m$ independent Bernoulli random variables with probability $\frac{1}{2}$. We have $\mathbb E[\area(z)]\geq \frac{m}{2}$, and Chernoff bound yields: 
\[\pr(\area(\gamma)< \frac{m}{3})\leq e^{-\frac{m}{128}}.\] 
There are exactly $2^{{n-1 \choose 2}}-1$ different non-zero elements in $Z_1(K_n)$, while the areas of two different cycles $z$ and $z'$ that share some edge are dependent random variables, it suffices to set $m=64 n^2$ so that the union bound yields
\[\pr(\exists z \in Z_1(K_n), \area(z)<\frac{m}{3})\leq  2^{-\frac{n}{2}}\]

Since this probability is strictly smaller than $1$, there exists some drawing such that the area of every cycle is at least $\frac{m}{3}$, which after perturbing and re-scaling by $\frac{1}{m}$, corresponds to all cycles having area at least $\frac{1}{3}-\epsilon$, for any given $\epsilon>0$.
\end{proof}

\begin{remark}
 If we only cared about graph cycles, i.e. subgraphs of $K_n$ in which each vertex has degree two, then it is enough to take $m=O(n \log n)$.
\end{remark}
\begin{remark}
There is nothing especial about $\frac{1}{4}$ or about $\frac{1}{3}$ here, at the cost of making $m$ larger, we can force all cycles to have area $<\frac{1}{2}-\epsilon$ for any $\epsilon>0$, and it is easy to see that for any complete topological graph there exists $z \in Z_1(G)$ with $\area(z)\leq \frac{1}{2}$.
\end{remark} 

\begin{problem}
The construction described above has $\Omega(n^6)$ crossings. Does there exists a cycle of small area in every drawing of $K_n$ such that every pair of edges intersect a constant number of times?
\end{problem}

\end{document}